\documentclass[12pt,reqno]{amsart}
\usepackage{amsfonts}
\usepackage{amssymb,color}
\usepackage{amssymb}

\usepackage[colorlinks=true]{hyperref}
\hypersetup{urlcolor=blue, citecolor=blue}

\usepackage[margin=3cm, a4paper]{geometry}

\def\norm#1{\|#1\|}

\newcommand{\e}{\varepsilon}

\newcommand{\De}{\Delta}

\newcommand{\p}{\partial}
\newcommand{\na}{\nabla}

\newcommand{\Del}[1]{}

\newcommand{\pt}{&}
\newcommand{\pr}{\\ &}

\newcommand{\prq}{\\ &\quad}

\numberwithin{equation}{section}

\newtheorem{thm}{Theorem}[section]

\newtheorem{lem}[thm]{Lemma}

\theoremstyle{remark}
\newtheorem{rem}{Remark}

\theoremstyle{remark}

\theoremstyle{definition}

\begin{document}

\title[Navier-Stokes equations]{
Local well-posedness of the incompressible Euler equations in $B^1_{\infty,1}$ and the inviscid limit of the Navier-Stokes equations}
\author[Z. Guo]{Zihua Guo}
\address{School of Mathematical Sciences, Monash University, Clayton VIC 3800, Australia}
\email{zihua.guo@monash.edu}

\author[J. Li]{Jinlu Li}
\address{Department of Mathematics, Sun Yat-sen University, Guangzhou, 510275, China}
\email{lijl29@mail2.sysu.edu.cn}

\author[Z. Yin]{Zhaoyang Yin}
\address{Department of Mathematics, Sun Yat-sen University, Guangzhou, 510275, China \& Faculty of Information Technology,
 Macau University of Science and Technology, Macau, China}
\email{mcsyzy@mail.sysu.edu.cn}
\begin{abstract}
We prove the inviscid limit of the incompressible Navier-Stokes equations in the same topology of Besov spaces as the initial data.  The proof is based on proving the continuous dependence of the Navier-Stokes equations uniformly with respect to the viscosity. To show the latter, we rely on some Bona-Smith type method in the $L^p$ setting. Our obtained result implies a new result that the Cauchy problem of the Euler equations is locally well-posed in the borderline Besov space $B^{\frac dp+1}_{p,1}(\mathbb{R}^d)$, $1\leq p\leq \infty$, $d\geq 2$, in the sense of Hadmard, which is an open problem left in recent works by Bourgain and Li in \cite{BL,BL1} and by Misio{\l}ek and Yoneda in \cite{MY,MY2, MY3}.
\end{abstract}

\maketitle


\section{Introduction}

In this article, we consider the incompressible Navier-Stokes equations
\begin{align}\label{eq:NS}
\begin{cases}
\p_t u+u\cdot \nabla u-\e \Delta u=-\nabla P,\\
\mathrm{div\,} u=0,\\
u(0,x)=u_0,
\end{cases}
\end{align}
where $u(t,x):[0,\infty)\times {\mathbb R}^d\to {\mathbb R}^d$ is the unknown velocity, $\e\geq 0$ is the viscocity parameter, and $P$ is the pressure term. When the viscocity vanishes, namely $\e=0$, then \eqref{eq:NS} reduces to the Euler equations for ideal incompressible fluid.
Both Navier-Stokes and Euler equations have been extensively studied and the
problems of global regularity for 3D equations are still challenging open problems. See \cite{BCD} for a survey of studies for both equations.

Formally, as $\e\to 0$, the solution of the Navier-Stokes equations converges to the solution of the Euler equation. To derive the convergence rigorously is the inviscid limit problem. This problem has been studied in many literatures. See for example \cite{Swann, Kato}, and \cite{CKV} for the inviscid limit on the bounded domain. In \cite{Majda}, Majda showed under the assumption $u_0\in H^s$ with $s>\frac{d}{2}+2$, the solutions $u_\e$ to \eqref{eq:NS} converge in $L^2$ norm as $\e\to 0$ to the unique solution of Euler equation and the convergence rate is of order $\e t$. In \cite{M}, Masmoudi proved the convergence in $H^s$ norm under the assumption $u_0\in H^s$ with $s>\frac{d}{2}+1$. In dimension two the results are global in time and were improved in \cite{HK} where the assumption is improved to $u_0\in B^2_{2,1}$ with convergence in $L^2$.  The two dimensional results were further generalized to other Besov spaces $B^{2/p+1}_{p,1}$ with convergence in $L^p$, see section 3.4 in \cite{MWZ}. In three dimension a similar result was proved in \cite{Wu} for axis-symmetric flows without swirl.  By interpolation with the uniform estimates, one can get the convergence in all intermediate spaces.  However the convergence in the same topology as the initial data (e.g. $B^{2/p+1}_{p,1}$ in 2D) was unknown and mentioned as an open problem in Remark 4.3 in \cite{MWZ}.

The purpose of this article is to study the inviscid limit in the same topology.  As a by-product, we obtain the continuous dependence for the Euler equations which was not proved in \cite{BCD} or other literatures that we are aware of.  The main result of this paper is

\begin{thm}\label{th1.1}
Let $d\geq 2$. Assume that $\e\in [0,1]$ and $(s,p,r)$ satifies
\begin{align}\label{eq:spr}
s>\frac{d}{p}+1, p\in [1,\infty], r\in (1,\infty) \quad \mbox{   or   } \quad \pt
s=\frac{d}{p}+1, p\in [1,\infty], r=1.
\end{align}
Then for any $R>0$, $u_0\in B_R=\{\phi\in B_{p,r}^s: \norm{\phi}_{B^{s}_{p,r}}\leq R,\ \mathrm{div\,}\ \phi=0\}$, there exists $T=T(R,s,p,r,d)>0$ such that the Navier-Stokes equation has a unique solution $u_\e=S_{T}^\e(u_0)\in C([0,T];B^s_{p,r})$. Moreover, we have

1) (Uniform bounds): there exists $C=C(R,s,p,r,d)>0$ such that
\begin{align}\label{UB1}
\norm{u_\e(t)}_{L_T^\infty B^s_{p,r}}\leq C, \quad \forall\ \e\in [0,1].
\end{align}
Moreover, if $u_0\in B^{\gamma}_{p,r}$ for some $\gamma>s$, then $\exists\ C_2=C_2(R, \gamma,s,p,r,d)>0$
\begin{align}\label{UB2}
\norm{u_\e(t)}_{L_T^\infty B^\gamma_{p,r}}\leq  C_2\norm{u_0}_{B^\gamma_{p,r}}.
\end{align}

2) (Uniform continuous dependence): the solution map $u_0\to S_T^\e(u_0)$ is continuous from $B_R$ to $C([0,T];B^s_{p,r})$ uniformly with respect to $\e$. Namely, $\forall\ \eta>0$, $\exists \delta=\delta(u_0,R,s,p,r,d)>0$ such that for any $\psi\in B_R$ with $\norm{\psi-u_0}_{B^s_{p,r}}<\delta$ then
\begin{align}\label{UB3}
\norm{S_T^\e(u_0)-S_T^\e(\psi)}_{L_T^\infty B^s_{p,r}}<\eta, \quad \forall\ \e\in [0,1].
\end{align}

3) (Inviscid limit): we have
\begin{align}\label{UB4}
\lim_{\e\to 0}\norm{S_T^\e(u_0)-S_T^0(u_0)}_{L_T^\infty B^s_{p,r}}=0.
\end{align}
\end{thm}

\begin{rem}
The novelty of the above theorem is part 2) and 3), while part 1) is classical.
For $d=2$, using the structures of the equations one can derive a global a-priori bound on $\norm{u(t)}_{B_{p,r}^s}$, and hence the above theorem holds for any $T$. For $d=3$, assuming axis-symmetry without swirl and an additional condition on the vorticity, we also have the above theorem for any $T$.
\end{rem}

\begin{rem}
In the case $s>\frac{d}{p}+1, p\in [1,\infty]$, $r=\infty$, the above theorem also holds for $u_0\in B^s_{p,\infty}$ assuming additionally
\begin{align}\label{eq:tail}
\lim_{j\to \infty}2^{js}\norm{\Delta_j u_0}_{p}=0.
\end{align}
Without \eqref{eq:tail} we have existence and uniqueness in $L_T^\infty B_{p,\infty}^s$ (see Theorem 7.1 in \cite{BCD}), however, no continuous dependence. The main reason is that $B_{p,\infty}^s$ functions can not be approximated by functions with compact Fourier support.
\end{rem}

\begin{rem}
Recently, Bourgain and Li in \cite{BL,BL1} employed a combination of Lagrangian and Eulerian techniques to obtain strong local ill-posedness results of the Euler equations in borderline Besov spaces $B^{\frac dp+1}_{p,r}$ for $1\leq p<\infty$ and $1<r\leq\infty$ when $d=2,3$.   Theorem 1.1 implies a new result that the Cauchy problem of the Euler equations is locally well-posed in the borderline Besov space $B^{\frac dp+1}_{p,1}(\mathbb{R}^d)$ for $1\leq p\leq \infty$, in particular $B^{1}_{\infty,1}(\mathbb{R}^d)$, $d\geq 2$, in the sense of Hadmard, which is an open problem left in recent works by Bourgain and Li in \cite{BL,BL1} and by Misio{\l}ek and Yoneda in \cite{MY,MY2, MY3}. 
In particular, the continuous dependence seems to be new and the arguments work for many other equations in fluid dynamics. This is a bit surprising since near $B^1_{\infty,1}$ there is some weak norm inflation phenomena, see \cite{MY,MY2,MY3}.

\end{rem}
The proof of the theorem is an application of the Bona-Smith method \cite{BS} but in the $L^p$ setting. The method is very useful in proving the continuity of the solution map especially when the solution map is not  Lipshitz or $C^k$ smooth. In our problem, the solution map of the Euler equation was known not to be locally Lipshitz (at least in the torus case, see \cite{HM}), hence one can not have
\begin{align}
\norm{S_T^0(\phi)-S_T^0(\psi)}_{L_T^\infty B^s_{p,r}}\leq C\norm{\phi-\psi}_{B^s_{p,r}}.
\end{align}
The essence of the Bona-Smith method is, however, to show for any $\phi\in B^s_{p,r}$
\begin{align}
\norm{S_T^0(S_N\phi)-S_T^0(\phi)}_{L_T^\infty B^s_{p,r}}\leq C\norm{S_N\phi-\phi}_{B^s_{p,r}},
\end{align}
where $S_N$ the frequency localization operator defined in Section 2.  With these estimates we can conclude the continuous dependence.

\section{Preliminaries}

In this section we collect some preliminary definitions and lemmas. For more details we refer the readers to \cite{BCD}.

Let $\chi: {\mathbb R}^d\to [0, 1]$ be a radial, non-negative,
smooth and radially decreasing function which is supported in $\mathcal{B}\triangleq \{\xi:|\xi|\leq \frac43\}$ and
$\chi\equiv 1$ for $|\xi|\leq \frac54$. Let $\varphi(\xi)=\chi(\frac{\xi}{2})-\chi(\xi)$. Then $\varphi$ is supported in the ring $\mathcal{C}\triangleq \{\xi\in\mathbb{R}^d:\frac 3 4\leq|\xi|\leq \frac 8 3\}$.
For $u \in \mathcal{S}'$, $q\in {\mathbb Z}$, we define the Littlewood-Paley operators: $\dot{\Delta}_q{u}=\mathcal{F}^{-1}(\varphi(2^{-q}\cdot)\mathcal{F}u)$, ${\Delta}_q{u}=\dot{\Delta}_q{u}$ for $q\geq 0$, ${\Delta}_q{u}=0$ for $q\leq -2$ and $\Delta_{-1}u=\mathcal{F}^{-1}(\chi \mathcal{F}u)$, and $S_q{u}=\mathcal{F}^{-1}\big(\chi(2^{-q}\xi)\mathcal{F}u\big)$.
Here we use ${\mathcal{F}}(f)$ or $\widehat{f}$ to denote
the Fourier transform of $f$.

We define the standard vector-valued Besov spaces $B^s_{p,r}$ and $\dot B^s_{p,r}$ of the functions $u:{\mathbb R}^d\to {\mathbb R}^d$ with finite norms which are defined by
\begin{align*}
\|u\|_{B^s_{p,r}}&\triangleq \big|\big|(2^{js}\|\Delta_j{u}\|_{L^p})_{j\in {\mathbb Z}}\big|\big|_{\ell^r},\\
\|u\|_{\dot{B}^s_{p,r}}&\triangleq \big|\big|(2^{js}\|\dot{\Delta}_j{u}\|_{L^p})_{j\in {\mathbb Z}}\big|\big|_{\ell^r}.
\end{align*}
Next we recall nonhomogeneous Bony's decomposition from \cite{BCD}.
$$uv=T_uv+T_vu+R(u,v),$$
with
\[T_uv\triangleq \sum\limits_j S_{j-1}u\Delta_j v, \quad R(u,v)\triangleq \sum_{j}\sum\limits_{|k-j|\leq1}\Delta_j u \Delta_k v.\]
This is now a standard tool for nonlinear estimates. Now we use Bony's decomposition to prove some nonlinear estimates which will be used for the estimate of pressure term.

\begin{lem}\label{lem:P}
Assume $(s,p,r)$ satisfies \eqref{eq:spr}. Then

1) there exists a constant $C$, depending only on $d,p,r,s$, such that for all $u,f\in B^s_{p,r}$ with $\mathrm{div\,} u=0$,
\[\|u\cdot \na f\|_{B^{s-1}_{p,r}}\leq C\|u\|_{B^{s-1}_{p,r}}\|f\|_{B^s_{p,r}}.\]

2) there exists a constant $C$, depending only on $d,p,r,s$, such that for all $u,v\in B^s_{p,r}$ with $\mathrm{div\,} u=\mathrm{div\,} v=0$,
\begin{align*}
\|\na(-\De)^{-1}\mathrm{div\,}(u\cdot \na v)\|_{B^s_{p,r}}&\leq C \big(\|u\|_{C^{0,1}}\|v\|_{B^s_{p,r}}+\|v\|_{C^{0,1}}\|u\|_{B^s_{p,r}}\big);\\
\|\na(-\De)^{-1}\mathrm{div\,}(u\cdot \na v)\|_{B^{s-1}_{p,r}}&\leq C \min(\|u\|_{B^{s-1}_{p,r}}\|v\|_{B^s_{p,r}},\|v\|_{B^{s-1}_{p,r}}\|u\|_{B^s_{p,r}}),
\end{align*}
where $\norm{f}_{C^{0,1}}=\norm{f}_{L^\infty}+\norm{\na f}_{L^\infty}$.
\end{lem}
\begin{proof}
This follows from a standard argument (e.g. see Lemmas 7.9-7.10, \cite{BCD} and Proposition 8, \cite{PP}) using Bony's decomposition and the fact that $u\cdot \na v=\mathrm{div\,}(v\otimes u)$ and $\mathrm{div\,}(u\cdot \na v)=\mathrm{div\,}(v\cdot \na u)$ when $\mathrm{div\,} u=\mathrm{div\,} v=0$.  We omit the details.
\end{proof}

We need an estimate for the transport-diffusion equation which is uniform with respect to the viscocity. Consider the following equation:
\begin{align}\label{eq:TDep}
\begin{cases}
\p_t f+v\cdot \nabla f-\e \Delta f=g,\\
f(0)=f_0,
\end{cases}
\end{align}
where $v:{\mathbb R}\times {\mathbb R}^d \to {\mathbb R}^d$, $f_0:{\mathbb R}^d\to {\mathbb R}^N$, and $g:{\mathbb R}\times {\mathbb R}^d\to {\mathbb R}^N$ are given.

\begin{lem}[Theorem 3.38, \cite{BCD}]\label{lem:TDe}
Let $1\leq p,r\leq \infty$. Assume that
\begin{align}
\sigma> -d \min(\frac{1}{p}, \frac{1}{p'}) \quad \mathrm{or}\quad \sigma> -1-d \min(\frac{1}{p}, \frac{1}{p'})\quad \mathrm{if} \quad \mathrm{div\,} v=0.
\end{align}
There exists a constant $C$, depending only on $d,p,r,\sigma$, such that for any smooth solution $f$ of \eqref{eq:TDep} and $t\geq 0$ we have
\begin{align}\label{ES2}
\sup_{s\in [0,t]}\norm{f(s)}_{B^{\sigma}_{p,r}}\leq Ce^{CV_{p}(v,t)}\big(\norm{f_0}_{B^\sigma_{p,r}}
+\int^t_0\norm{g(\tau)}_{B^{s}_{p,r}}d \tau\big),
\end{align}
with
\begin{align*}
V_{p}(v,t)=
\begin{cases}
\int_0^t \norm{\nabla v(s)}_{B^{\frac{d}{p}}_{p,\infty}\cap L^\infty}ds,\quad \mathrm{if} \quad \sigma<1+\frac{d}{p},\\
\int_0^t \norm{\nabla v(s)}_{B^{\sigma-1}_{p,r}}ds, \quad \quad \mathrm{if} \quad \sigma>1+\frac{d}{p}\ \mathrm{or}\ \{\sigma=1+\frac{d}{p} \mbox{ and } r=1\}.
\end{cases}
\end{align*}
If $f=v$, then for all $\sigma>0$ ($\sigma>-1$, if $\mathrm{div\,} v=0$), the estimate \eqref{ES2} holds with
\[V_{p}(t)=\int_0^t \norm{\nabla v(s)}_{L^\infty}ds.\]
\end{lem}

\section{Proof of Theorem \ref{th1.1}}
In this section we prove Theorem \ref{th1.1}. Assume $(s,p,r)$ satisfies the conditions in Theorem \ref{th1.1}. For fixed $\e>0$, by classical results we known there exists $T_\e=T(\norm{u_0}_{B^s_{p,r}},\e)>0$ such that the Navier-Stokes system \eqref{eq:NS} has a unique solution $u$ in $\mathcal{C}([0,T_\e];B^s_{p,r})$.

{\bf Step 1.} We show: $\exists\ T=T(\norm{u_0}_{B^s_{p,r}})>0$ such that $T_\e\geq T$. Moreover, \eqref{UB1} and \eqref{UB2} hold.

By the relation $P=P(u):=(-\De)^{-1}\mathrm{div\,}(u\cdot\na u)$, we have the following estimates (see Lemma \ref{lem:P}):
\begin{align}\label{P1}
\|\nabla P\|_{B^s_{p,r}}\leq C(\|u\|_{L^\infty}+\|\na u\|_{L^\infty})\|\na u\|_{B^s_{p,r}}.
\end{align}
By Lemma \ref{lem:TDe}  and \eqref{P1}, we have
\begin{align*}
\norm{u(t)}_{B^s_{p,r}}&\leq C e^{CV_p(u,t)} \big(\norm{u_0}_{B^s_{p,r}}+\int^t_0\norm{\nabla P(\tau)}_{B^s_{p,r}}d \tau\big)
\pr \leq Ce^{CV_p(u,t)}\big(\norm{u_0}_{B^s_{p,r}}+\int^t_0\norm{u(\tau)}^2_{B^s_{p,r}}d \tau\big).
\end{align*}
Since $V_p(u,t)\leq \int_0^t \norm{u(\tau)}_{B^s_{p,r}}d\tau$, then by continuity arguments there exists $T=T(\norm{u_0}_{{B^s_{p,r}}})>0$ such that
\begin{align*}
\|u(t)\|_{B^s_{p,r}}\leq C, \quad t\in [0,T].
\end{align*}
Similarly,
\begin{align*}
\norm{u(t)}_{B^\gamma_{p,r}}&\leq C e^{CV_p(u,t)} \big(\norm{u_0}_{B^\gamma_{p,r}}+\int^t_0\norm{\nabla P(\tau)}_{B^\gamma_{p,r}}d \tau\big)\pr
\leq Ce^{CV_p(u,t)}\big(\norm{u_0}_{B^\gamma_{p,r}}+\int^t_0\norm{u(\tau)}_{B^s_{p,r}}\norm{u(\tau)}_{B^\gamma_{p,r}}d \tau\big)
\end{align*}
and Gronwall's inequality, we obtain
\begin{align*}
\norm{u(t)}_{B^\gamma_{p,r}}\leq Ce^{C\int_0^t \norm{u(\tau)}_{B^s_{p,r}} d\tau}\norm{u_0}_{B^\gamma_{p,r}}\leq C\norm{u_0}_{B^\gamma_{p,r}}.
\end{align*}

{\bf Step 2.} We show that the solution map of \eqref{eq:NS} is continuous in a uniform way with respect to $\e\in [0,1]$.

First we show for any $\phi,\psi\in B_R$
\begin{align}\label{eq:diff}
\norm{S_T^\e(\phi)-S_T^\e(\psi)}_{L_T^\infty B^{s-1}_{p,r}}\leq C \norm{\phi-\psi}_{B^{s-1}_{p,r}}, \quad \forall \ \e\in (0,1].
\end{align}
Indeed, denote $\Phi=S_T^\e(\phi)$ and $\Psi=S_T^\e(\psi)$.
By Step 1, we have
\begin{align}\label{eq:unibd}
\norm{\Phi}_{L_T^\infty B^s_{p,r}}+\norm{\Psi}_{L_T^\infty B^s_{p,r}}\leq C.
\end{align}
Let $W=\Phi-\Psi$. Then
\begin{align*}
\begin{cases}
\p_tW+\Psi\cdot \na W+W\cdot \na \Phi-\e\De W=-\na(P(\Phi)-P(\Psi)),\\
\mathrm{div\,} \Phi=\mathrm{div\,} \Psi=0,\\
W(0,x)=\phi-\psi.
\end{cases}
\end{align*}
Since $\na(P(\Phi)-P(\Psi))=\na (-\Delta)^{-1}\mathrm{div\,} (W\cdot \na \Phi+\Psi\cdot \na W)$, then by Lemma \ref{lem:P}
\begin{align}\label{E4}
\|\na(P(\Phi)-P(\Psi))\|_{B^{s-1}_{p,r}}\leq C\|W\|_{B^{s-1}_{p,r}}(\|\Phi\|_{B^s_{p,r}}
+\|\Psi\|_{B^s_{p,r}}).
\end{align}
Then by Lemma \ref{lem:TDe}, Lemma \ref{lem:P} and \eqref{E4} we get
\begin{align*}
\|W(t)\|_{B^{s-1}_{p,r}}&\leq Ce^{CV_p(\Psi,T)}\big(\|W(0)\|_{B^{s-1}_{p,r}}+\int^t_0 \|W\cdot\na \Phi\|_{B^{s-1}_{p,r}}d\tau\\
&\qquad \qquad\qquad\qquad\qquad+\int_0^t\|\na(P(\Phi)-P(\Psi))\|_{B^{s-1}_{p,r}}d \tau\big)\nonumber \pr
\leq Ce^{CV_p(\Psi,T)}\big(\|W(0)\|_{B^{s-1}_{p,r}}+\int^t_0\|W\|_{B^{s-1}_{p,r}}(\|\Phi\|_{B^s_{p,r}}+\|\Psi\|_{B^s_{p,r}})d \tau\big),
\end{align*}
which by Gronwall's inequality and \eqref{eq:unibd} implies
\begin{align}\label{E7-1}
\|W(t)\|_{B^{s-1}_{p,r}}\leq C\|W(0)\|_{B^{s-1}_{p,r}},
\end{align}
and thus \eqref{eq:diff} is proved.

Next we show for any $\phi\in B_R$
\begin{align}\label{eq:diffN}
\norm{S_T^\e(S_{ N} \phi)-S_T^\e(\phi)}_{L_T^\infty B^s_{p,r}}\leq C \norm{S_{ N} \phi-\phi}_{B^s_{p,r}}, \quad \forall \ \e\in [0,1].
\end{align}
Indeed, denote $v_N=S_T^\e(S_{N} \phi)$ and $v=S_T^\e(\phi)$. Then by Step 1, we have $\norm{v}_{L_T^\infty B^{s}_{p,r}}\leq C$ and
\begin{align}
\norm{v_N}_{L_T^\infty B^{s+k}_{p,r}}\leq C\norm{S_N\phi}_{B^s_{p,r}} \leq C2^{Nk}, \quad k=0,1,2.
\end{align}
Let $w_N=v_N-v$. Then we have
\begin{align*}
\begin{cases}
\p_tw_N+v\cdot \na w_N+w_N\cdot \na v_N-\e\De w_N=-\na(P(v_N)-P(v)),\\
\mathrm{div\,} v=\mathrm{div\,} v_N=0,\\
w_N(0,x)=S_{ N} \phi-\phi.
\end{cases}
\end{align*}
Since $\na(P(v_N)-P(v))=\na (-\De)^{-1}\mathrm{div\,}(v\cdot \na w_N+w_N\cdot \na v_N)$,
by Lemma \ref{lem:P} we get
\begin{align*}
\|\na(P(v_N)-P(v))\|_{B^s_{p,r}}\leq& C\|w_N\|_{B^s_{p,r}}(\|v_N\|_{B^s_{p,r}}
+\|v\|_{B^s_{p,r}})
\end{align*}
and
\begin{align*}
\|w_N\cdot\na v_N\|_{B^s_{p,r}}\leq& C\big(\|w_N\|_{L^\infty}\|\na v_N\|_{B^s_{p,r}}
+\|w_N\|_{B^s_{p,r}}\|\na v_N\|_{L^\infty}\big)\\
\leq&C\big(\|w_N\|_{B^{s-1}_{p,r}}\|v_N\|_{B^{s+1}_{p,r}}
+\|w_N\|_{B^s_{p,r}}\|v_N\|_{B^s_{p,r}}\big).
\end{align*}
Therefore, by Lemma \ref{lem:TDe} and the above estimates we get
\begin{align*}
\|w_N(t)\|_{B^s_{p,r}}\leq& Ce^{CV_p(v,T)}\big(\|w_N(0)\|_{B^s_{p,r}}+\int^t_0\|w_N\cdot\na v_N\|_{B^s_{p,r}}+\|\na(P(v_N)-P(v))\|_{B^s_{p,r}}d \tau\big)\\
\leq& Ce^{CV_p(v,T)}\big(\|w_N(0)\|_{B^s_{p,r}}+\int^t_0\|w_N\|_{B^s_{p,r}}(\|v_N\|_{B^s_{p,r}}+\|v\|_{B^s_{p,r}})d\tau\\
&\qquad\qquad\qquad\qquad+\int_0^t\|w_N\|_{B^{s-1}_{p,r}}\|v_N\|_{B^{s+1}_{p,r}} d \tau\big)\\
\leq& C\big(\|w_N(0)\|_{B^s_{p,r}}+\int^t_0C\|w_N\|_{B^s_{p,r}}+C2^N\|w_N(0)\|_{B^{s-1}_{p,r}} d\tau\big)\\
\leq& C\big(\|w_N(0)\|_{B^s_{p,r}}+C\int^t_0\|w_N\|_{B^s_{p,r}} d\tau\big).
\end{align*}
Using Gronwall's inequality we prove \eqref{eq:diffN}.

Now we prove the continuous dependence in $B^s_{p,r}$. For any $\phi,\psi\in B_R$ we have
\begin{align*}
&\norm{S_T^\e(\phi)-S_T^\e(\psi)}_{L_T^\infty B^s_{p,r}}\\
\leq &\norm{S_T^\e(\phi)-S_T^\e(S_N\phi)}_{L_T^\infty B^s_{p,r}}+\norm{S_T^\e(\psi)-S_T^\e(S_N\psi)}_{L_T^\infty B^s_{p,r}}\\
&\quad+\norm{S_T^\e(S_N\phi)-S_T^\e(S_N\psi)}_{L_T^\infty B^s_{p,r}}\\
\leq&C(\norm{\phi-S_N\phi}_{B^s_{p,r}}+\norm{\phi-\psi}_{B^s_{p,r}})\\
&\quad+C\norm{S_T^\e(S_N\phi)-S_T^\e(S_N\psi)}_{L_T^\infty B^{s-1}_{p,r}}^{1/2}\norm{S_T^\e(S_N\phi)-S_T^\e(S_N\psi)}_{L_T^\infty B^{s+1}_{p,r}}^{1/2}\\
\leq&C(\norm{\phi-S_N\phi}_{B^s_{p,r}}+\norm{\phi-\psi}_{B^s_{p,r}})+C2^{N/2}\norm{\phi-\psi}_{B^s_{p,r}}^{1/2}.
\end{align*}
With the above estimate we obtain the continuous dependence.

{\bf Step 3.} We finally prove the inviscid limit of the Navier-Stokes system \eqref{eq:NS}.

To obtain the result \eqref{UB4}, we decompose the left term of \eqref{UB4} into
\begin{align}\label{EE}
\norm{S_T^\e(u_0)-S_T^0(u_0)}_{L_t^\infty B^s_{p,r}}&\leq \norm{S_T^\e(S_{N}u_0)-S_T^0(S_Nu_0)}_{L_t^\infty B^s_{p,r}}\prq
+\norm{S_T^\e(S_{N}u_0)-S_T^\e(u_0)}_{L_t^\infty B^s_{p,r}}\nonumber\prq
+\norm{S_T^0(S_{N}u_0)-S_T^0(u_0)}_{L_t^\infty B^s_{p,r}}\nonumber.
\end{align}
We set $u^\e_N=S_T^\e(S_Nu_0)$, $u_N=S_T^0(S_Nu_0)$ and $w^\e_N=u^\e_N-u_N$ and have
\begin{align}
\begin{cases}
\p_tw^\e_N+u_N\cdot\na w^\e_N+w^\e_N\cdot\na u^\e_N=-\na(P(u^\e_N)-P(u_N))+\e\Delta u^\e_N,\\
\mathrm{div\,} u^\e_N=\mathrm{div\,} u_N=0,\\
w^\e_N(0,x)=0.
\end{cases}
\end{align}
Similarly as Step 2, we have
\begin{align*}
\|w^\e_N(t)\|_{B^{s-1}_{p,r}}
\leq& C\int^t_0\|w^\e_N(\tau)\|_{B^{s-1}_{p,r}}\big(\|u^\e_N(\tau)\|_{B^s_{p,r}}+\|u_N(\tau)\|_{B^s_{p,r}}\big)d \tau+C\e2^{N},
\end{align*}
which implies
\begin{align}\label{E14}
\|w^\e_N(t)\|_{B^{s-1}_{p,r}}\leq C\e2^{N}.
\end{align}
Moreover,
\begin{align}\label{E12}
\|w^\e_N(t)\|_{B^s_{p,r}}
\leq& Ce^{V_p(u_N,T)}\Big(\int^t_0\big(\|w^\e_N(\tau)\|_{B^s_{p,r}}\|u^\e_N(\tau)\|_{B^s_{p,r}}+\|w^\e_N(\tau)\|_{B^s_{p,r}}\|u_N(\tau)\|_{B^s_{p,r}}\nonumber \prq +\|w^\e_N(\tau)\|_{B^{s-1}_{p,r}}\|u^\e_N(\tau)\|_{B^{s+1}_{p,r}}\big)d \tau+ \e\int^t_0\|u^\e_N(\tau)\|_{B^{s+2}_{p,r}}d \tau \Big)\nonumber\\
\leq& C\int^t_0\|w^\e_N(\tau)\|_{B^s_{p,r}}\big(\|u^\e_N(\tau)\|_{B^s_{p,r}}+\|u_N(\tau)\|_{B^s_{p,r}}\big)d \tau
+C\e2^{2N},
\end{align}
which along with Gronwall's inequality leads to
\begin{align}\label{E15}
\|w^\e(t)\|_{B^s_{p,r}}&\leq C\e2^{2N}.
\end{align}
Therefore, combining \eqref{eq:diffN}, \eqref{EE} and \eqref{E15}, we have
\begin{align*}
\norm{S_T^\e(u_0)-S_T^0(u_0)}_{L_t^\infty B^s_{p,r}}&\leq  C(\norm{u_0-S_Nu_0}_{B^s_{p,r}}+\e2^{2N}).
\end{align*}
This completes the proof of \eqref{UB4}.

\subsection*{Acknowledgements}
Z. Y. was partially supported by NNSFC (No. 11671407),  FDCT (No. 098/2013/A3), Guangdong Special Support Program (No. 8-2015), and the key project of NSF of  Guangdong province (No. 2016A030311004).


\end{document}